\documentclass[reqno]{article}

\usepackage{bbm, amsmath, amsfonts, amssymb, amsthm}
\usepackage{xparse}
\usepackage{cite}

\newtheorem{lemma}{Lemma}
\newtheorem{proposition}{Proposition}
\newtheorem{theorem}{Theorem}
\newtheorem{corollary}{Corollary}[section]
\newtheorem{referencetheorem}{Theorem}

\newtheorem{referencelemma}[referencetheorem]{Lemma}

\theoremstyle{definition}\newtheorem*{definition}{Definition}
\newcommand{\D}{\mathcal{D}}
\newcommand{\n}[1]{\left\| #1 \right\|}
\newcommand{\br}[1]{\left( #1 \right)}

\newcommand{\m}[1]{\left| #1 \right|}

\NewDocumentCommand\sm{O{j}O{1}O{\infty}}{\sum\limits_{#1=#2}^{#3}}
\NewDocumentCommand\seq{O{j}O{1}O{\infty}m}{\left\{#4_{#1}\right\}_{#1=#2}^{#3}}

\begin{document}
\title{On the Approximate Weak Chebyshev Greedy Algorithm
	in Uniformly Smooth Banach Spaces}
\author{A.\,V.~Dereventsov}
\date{}
\maketitle

\begin{abstract}
We study greedy approximation in uniformly smooth Banach spaces.
The Weak Chebyshev Greedy Algorithm (WCGA), introduced and 
studied in~\cite{t_gabs}, is defined for any Banach space $X$ and 
a dictionary $\D$, and provides nonlinear $n$-term approximation 
with respect to $\D$.
In this paper we study the Approximate Weak Chebyshev Greedy Algorithm 
(AWCGA)~-- a modification of the WCGA that was studied in~\cite{t_gtabsa}. 
In the AWCGA we are allowed to calculate $n$-term approximation with
a perturbation in computing the norming functional and 
a relative error in calculating the approximant. 
Such permission is natural for the numerical 
applications and simplifies realization of the algorithm.
We obtain conditions that are necessary and sufficient for the convergence of the 
AWCGA for any element of $X$. In particular, we show that 
if perturbations and errors are from $\ell_1$ space then the conditions 
for the convergence of the AWCGA are the same as for the WCGA.
For specifically chosen perturbations and errors we estimate the rate of 
convergence for any element $f$ from the closure of the convex hull of $\D$ 
and demonstrate that in special cases the AWCGA performs as well as the WCGA.

\end{abstract}

\section{Introduction}\label{introduction}

This paper is devoted to the problem of greedy approximation 
in Banach spaces. This problem was extensively researched by 
V.\,N.~Temlyakov (see, for instance,
~\cite{t_gabs},~\cite{t_gtabsa},~\cite{t_ga}).
For the Weak Chebyshev Greedy Algorithm (WCGA) in uniformly smooth
Banach spaces, the sufficient conditions for convergence and
the rate of approximation were obtained in~\cite{t_gabs}.
In this article we study the Approximate Weak Chebyshev Greedy 
Algorithm~-- the modification of the WCGA with perturbations in computing 
the norming functionals and relative errors in calculating the approximants.
This algorithm was analyzed in~\cite{t_gtabsa} and the sufficient
conditions for convergence and an estimate for
the rate of convergence in a special case were obtained.
In this paper we apply the technique used in~\cite{t_gtabsa}. 
We prove weaker sufficient conditions for convergence, 
show that they are sharp, obtain an
estimate for the rate of approximation, and compare them 
with known results.

This paper is devoted to the Banach space setting.
Let $X$ be a real Banach space.
A dictionary is a set $\D$ of elements from $X$ if 
$\n{g} = 1$ for each element $g \in \D$ and 
$\overline{\text{span}\,\D} = X$.
For convenience of notations we consider symmetric dictionaries,
i.e. such dictionaries that $g \in \D$ implies $-g \in \D$. 
By $A_0(\D)$ we denote all the linear combinations of the elements of
a dictionary $\D$, and by $A_1(\D)$ we denote the closure of the convex hull of 
a dictionary $\D$.\\
For any non-zero element $f \in X$ let $F_f$ denote a norming functional of $f$, 
i.e. such a functional that $\n{F_f} = 1$ and $\n{F_f(f)} = \n{f}$.
The existence of such a functional is guaranteed by the Hahn-Banach theorem.\\
A weakness sequence $\seq[n]{t}$ is a sequence of 
real numbers $t_n$ such that $0 \leq t_n \leq 1$ for any $n \geq 1$.
A perturbation sequence $\seq[n][0]{\delta}$ is a sequence of 
real numbers $\delta_n$ such that $0 \leq \delta_n \leq 1$ for any $n \geq 0$.
An error sequence $\seq[n]{\eta}$ is a sequence of 
real numbers $\eta_n$ such that $\eta_n \geq 0$ for any $n \geq 1$.
By $\eta_0$ we denote the least upper bound of $\seq[n]{\eta}$.

For a Banach space $X$, a dictionary $\D$, and an element $f \in X$,
the Approximate Weak Chebyshev Greedy Algorithm (AWCGA) with 
a weakness sequence $\seq[n]{t}$, a perturbation sequence $\seq[n][0]{\delta}$,
and an error sequence $\seq[n]{\eta}$ is defined as follows.

\begin{definition}[AWCGA]
Set $f_0 = f$ and for $n \geq 1$
\begin{enumerate}
\item take any functional $F_{n-1}$ satisfying
\begin{equation}\label{awcga_F_n}
\n{F_{n-1}} \leq 1
\quad\text{and}\quad
F_{n-1}(f_{n-1}) \geq (1 - \delta_{n-1}) \n{f_{n-1}};
\end{equation}
and find any $\phi_n \in \D$ such that
\begin{equation}\label{awcga_phi_n}
F_{n-1}(\phi_n) \geq t_n \sup\limits_{g \in \D} F_{n-1}(g);
\end{equation}
\item for $\Phi_n = \text{span}\{\phi_j\}_{j=1}^{n}$ 
denote $E_n = \inf\limits_{G\in\Phi_n}\n{f - G}$
and find any $G_n \in \Phi_n$ satisfying
\begin{equation}\label{awcga_G_n}
\n{f - G_n} \leq (1 + \eta_n) E_n;
\end{equation}
\item set $f_n = f - G_n$.
\end{enumerate}
\end{definition}

Note that if the supremum is not attained, one can select $t_n < 1$
and proceed with the algorithm.
In particular, if $t_n < 1$ for any $n$ then the AWCGA can be realized
for all Banach spaces and all dictionaries.
Moreover, there might be several elements satisfying 
conditions~\eqref{awcga_F_n}--\eqref{awcga_G_n}, so the algorithm does 
not guarantee uniqueness of realization. 
We say that the AWCGA converges if $\n{f_n} \to 0$ for any realization 
of the algorithm. Conversely, the AWCGA diverges if there exists such a 
realization that $\n{f_n} \not\to 0$.

Recall that a Banach space is smooth if for any non-zero $f$
the norming functional $F_f$ is unique.
For a Banach space $X$ the modulus of smoothness $\rho(u)$ is defined by
\begin{equation}\label{modulusofsmoothness}
\rho(u) = \sup_{\n{x}=\n{y}=1} \frac{\n{x+uy} + \n{x-uy}}{2} - 1.
\end{equation}
A Banach space is uniformly smooth if $\rho(u) = o(u)$ as $u \to 0$.
We say that the modulus of smoothness $\rho(u)$ is of power type $1 \leq q \leq 2$ 
if $\rho(u) \leq \gamma u^q$ for some $\gamma > 0$.
It follows from definition~\eqref{modulusofsmoothness} that 
every Banach space has a modulus of smoothness of power type $1$ and
that every Hilbert space has a modulus of smoothness of power type $2$.
Denote by $\mathcal{P}_q$ the class of all Banach spaces with the 
modulus of smoothness of nontrivial power type $1 < q \leq 2$.
In particular, it is known (see Lemma~B.1 from~\cite{donahue_et_al})
that the modulus of smoothness $\rho_p(u)$ of $L_p$ space satisfies
$$
\rho_p(u) \leq \left\{
\begin{array}{ll}
\frac{1}{p} \, u^p & 1 < p \leq 2 \\
\frac{p-1}{2} \, u^2 & 2 \leq p < \infty
\end{array}
\right.,
$$
hence $L_p \in \mathcal{P}_q$, where $q = \min\{p; 2\}$.

We study approximation in uniformly smooth Banach spaces. 
The reason for such restriction is supported in~\cite{donahue_et_al}, where
it is shown that convergence of a nonlinear approximation 
strongly depends on the norm smoothness of the space, and that if a space is not smooth,
the convergence of incremental greedy approximation cannot be guaranteed for all 
elements $f$ from $X$ (Theorem~3.1); however, if the space is uniformly smooth, 
the convergence is guaranteed for all elements $f$ from $X$ (Theorem~3.4). 
Additionally, it is shown in~\cite{dubinin_gaa} that the smoothness 
of a space is required for the convergence of the WCGA.
For completeness, we prove this result in Proposition~\ref{proposition_not_smooth} 
from section~\ref{convergence_of_the_awcga}.

We note, however, that the uniform smoothness of a space is not a necessary 
condition for the convergence: it is shown in~\cite{dkt_csgabs} 
that if $X$ is a separable reflexive Banach space, then 
$X$ admits an equivalent norm for which the WCGA converges
for all dictionaries $\D$ and all elements $f \in X$ 
(see Theorem~2.6 from~\cite{dkt_csgabs}).
On the other hand, spaces which admit 
an equivalent uniformly smooth norm are precisely super-reflexive spaces 
(see e.g.~\cite{enflo_bswcgeuscn}).
Therefore, any separable reflexive Banach space $X$, which is 
not super-reflexive, admits an equivalent norm for which the WCGA converges
for all dictionaries and elements; however, this norm would not
be uniformly smooth. An example of a separable reflexive but not 
super-reflexive Banach space can be found in~\cite{beauzamy_ibstg}.

The goal of this paper is to establish sufficient and
necessary conditions for convergence of the AWCGA
for all uniformly smooth Banach spaces with the modulus of smoothness
of a nontrivial power type, all dictionaries, and all elements of the space.
We understand the necessity of a condition in the following sense:
if the given condition does not hold, there exists a Banach space
$X \in \mathcal{P}_q$, a dictionary $\D$, and an element $f \in X$
such that there is a realization of the AWCGA of $f$ that does not converge to $f$.

We start with the known result, which
states the sufficient condition for the convergence 
of the WCGA (see Corollary~2.1 from~\cite{t_gabs}).

\begin{referencetheorem}\label{t_wcga_convergence}
Let $X \in \mathcal{P}_q$ be a Banach space and $\D$ be any dictionary. 
Let $\seq[n]{t}$ be a weakness sequence.
Assume that
\begin{equation}\label{t_wcga_sum_t^p=infty}
\sm[n] t_n^p = \infty,
\end{equation}
where $p = q/(q-1)$.
Then the WCGA converges for any $f \in X$.
\end{referencetheorem}

It is shown (see Proposition~2.1 from~\cite{t_gabs}) that
condition~\eqref{t_wcga_sum_t^p=infty} is sharp,
i.e. Theorem~\ref{t_wcga_convergence} gives necessary and sufficient
conditions of the convergence of the WCGA for all Banach spaces
$X \in \mathcal{P}_q$, all dictionaries $\D$, and all elements $f \in X$.

The next result states the sufficient condition for the convergence 
of the AWCGA (see Corollary~2.2 from~\cite{t_gtabsa}).

\begin{referencetheorem}\label{t_awcga_convergence}
Let $X \in \mathcal{P}_q$ be a Banach space and $\D$ be any dictionary.
Let $\seq[n]{t}$ be a weakness sequence,
$\seq[n][0]{\delta}$ be a perturbation sequence,
and $\seq[n]{\eta}$ be a bounded error sequence.
Assume that
\begin{align}
\label{t_awcga_sum_t^p=infty}
&\sm[n] t_n^p = \infty,\\
\label{t_awcga_delta=o(t^p)}
&\delta_n = o(t_{n+1}^p),\\
\label{t_awcga_eta=o(t^p)}
&\eta_n = o(t_n^p),
\end{align}
where $p = q/(q-1)$.
Then the AWCGA converges for any $f \in X$.
\end{referencetheorem}

This paper proposes that we weaken 
conditions~\eqref{t_awcga_sum_t^p=infty}--\eqref{t_awcga_eta=o(t^p)}
by imposing them only on subsequences. We show that
in this case the specified conditions are necessary 
for the convergence of the AWCGA.
The following theorem is the main result of this paper.

\begin{theorem}\label{awcga_convergence}
Let $\seq[n]{t}$ be a weakness sequence,
$\seq[n][0]{\delta}$ be a perturbation sequence,
and $\seq[n]{\eta}$ be a bounded error sequence.
Then the AWCGA converges for any Banach space $X \in \mathcal{P}_q$, 
any dictionary $\D$, and any element $f \in X$ if and only if
there exists a subsequence $\{n_k\}_{k=1}^{\infty}$ such that
\begin{align}
\label{awcga_sum_t^p=infty}
&\sm[k] t_{n_k+1}^p = \infty,
\\
\label{awcga_delta=o(t^p)}
&\delta_{n_k} = o(t_{n_k+1}^p),
\\
\label{awcga_eta=o(t^p)}
&\eta_{n_k} = o(t_{n_k+1}^p),
\end{align}
where $p = q/(q-1)$.
\end{theorem}

The particular advantage of this subsequence approach is that once a 
subsequence $\seq[k]{n}$ satisfying 
conditions~\eqref{awcga_sum_t^p=infty}--\eqref{awcga_eta=o(t^p)} is found,
only the choice of elements $\phi_{n_k}$ is essential for convergence, 
so arbitrary elements $\phi_j$ can be chosen on other steps.

In the special case $t_n = t > 0$ for all $n \geq 1$, 
Theorem~\ref{t_awcga_convergence} provides
that the condition $\lim\limits_{n\to\infty} \br{\delta_n + \eta_n} = 0$ 
guarantees the convergence of the AWCGA. 
Theorem~\ref{awcga_convergence} shows that in this case the weaker condition 
$\liminf\limits_{n\to\infty} \br{\delta_n + \eta_n} = 0$
is a necessary and sufficient condition for the convergence of the AWCGA, 
i.e. the following theorem holds.

\begin{theorem}\label{awcga_convergence_liminf_eta=0_liminf_delta=0} 
Let $\seq[n]{t}$ be a weakness 
sequence with $\liminf\limits_{n\to\infty} t_n > 0$,
$\seq[n][0]{\delta}$ be a perturbation sequence, 
and $\seq[n]{\eta}$ be a bounded error sequence.
Then the AWCGA converges for any Banach space $X \in \mathcal{P}_q$, 
any dictionary $\D$, and any element $f \in X$ if and only if 
$$
\liminf\limits_{n\to\infty} \br{\delta_n + \eta_n} = 0.
$$
\end{theorem}

Note that by using weaker restrictions on the modulus of smoothness 
and by applying the technique from Theorem~2.2 from~\cite{t_gtabsa},
the previous theorem can be stated for any uniformly smooth Banach space. 

Combining Theorem~\ref{awcga_convergence} with Lemma~\ref{lemma_sequences_l_1}
proven in section~\ref{convergence_of_the_awcga},
we obtain the following theorem, which states that necessary and sufficient 
conditions of the convergence of the AWCGA in case when 
perturbation and error sequences are from $\ell_1$ space, are the same as for the WCGA.

\begin{theorem}\label{awcga_convergence_eta_delta_in_l_1}
Let $\seq[n]{t}$ be a weakness 
sequence, $\seq[n][0]{\delta} \in \ell_1$ be a perturbation sequence,
and $\seq[n]{\eta} \in \ell_1$ be an error sequence.
Then the AWCGA converges for any Banach space $X \in \mathcal{P}_q$, 
any dictionary $\D$, and any element $f \in X$ if and only if 
$$
\sm[n] t_n^p = \infty,
$$
where $p = q/(q-1)$.
\end{theorem}

We note that with the minor changes in the proof of 
Theorem~\ref{t_awcga_convergence}, condition~\eqref{t_awcga_eta=o(t^p)} can be 
replaced with $\eta_n = o(t_{n+1}^p)$. It is easy to see that this modified 
version follows from Theorem~\ref{awcga_convergence}. 
In section~\ref{convergence_of_the_awcga} we show that Theorem~\ref{awcga_convergence}
implies the unmodified version of Theorem~\ref{t_awcga_convergence} as well.

\section{Convergence of the AWCGA}\label{convergence_of_the_awcga}

First, we justify the restrictions imposed on a Banach space $X$.
The following proposition shows that if $X$ is not smooth, then 
for some dictionary $\D$ and some function $f$, the WCGA of $f$ 
does not converge even if $f$ is a finite linear combination 
of the elements of the dictionary.

\begin{proposition}\label{proposition_not_smooth}
Let $X$ be a nonsmooth Banach space. Then there exists a dictionary $\D$
and an element $f \in A_0(\D)$ such that WCGA of $f$ with any weakness 
sequence $\seq[n]{t}$ does not converge to $f$.
\end{proposition}
\begin{proof}
Since $X$ is not smooth, there exists an element $f$ from the unit sphere 
of $X$ with two norming functionals $F$ and $F'$ such that $F \not\equiv F'$.
Then there exists an element $g\in X$ such that $F(g) \ne F'(g)$.
Without loss of generality assume that $F(g) > F'(g)$.
Denote
\begin{equation}
\label{not_smooth_g_0_g_1}
g_0 = \alpha_0 \br{g - \frac{F(g)+F'(g)}{2}f} 
\quad\text{and}\quad
g_1 = \alpha_1 \br{g - F(g)f},
\end{equation}
where $\alpha_0 = \n{g - \frac{F(g)+F'(g)}{2}f}^{-1}$
and $\alpha_1 = \n{g - F(g)f}^{-1}$. Note that $F(g_0) > 0$ and $F'(g_0) < 0$.
Let $\{e_j\}_{j \in \Lambda}$ be a dictionary in $X$. 
Consider the set of indices 
$$
\Lambda' = \left\{j \in \Lambda : e_j - \frac{F(e_j)}{F(g_0)}g_0 \ne 0\right\}.
$$
Define for any $j \in \Lambda'$
\begin{equation}
\label{not_smooth_e'_j}
e'_j = \beta_j \br{e_j - \frac{F(e_j)}{F(g_0)}g_0},
\quad\text{where}\quad
\beta_j = \n{e_j - \frac{F(e_j)}{F(g_0)}g_0}^{-1}. 
\end{equation}
We claim that $\D = \{\pm g_0, \pm g_1\} \cup \{\pm e'_j\}_{j \in \Lambda'}$
is a dictionary as well. Indeed, take any $h \in X$ and pick any $\epsilon > 0$. 
Then, since $\{e_j\}_{j\in\Lambda}$ is a dictionary,
there exist coefficients $\{a_j\}_{j\in\Lambda}$ such that
$\n{h - \sum\limits_{j \in \Lambda} a_j e_j} < \epsilon$.
Since 
\begin{align*}
\sum_{j \in \Lambda} a_j e_j
&= \sum_{j \in \Lambda'} a_j \br{\beta^{-1}_j e'_j + \frac{F(e_j)}{F(g_0)}g_0}
+ \sum_{j \in \Lambda \setminus \Lambda'} a_j \frac{F(e_j)}{F(g_0)}g_0
\\
&= \br{\sum_{j \in \Lambda} a_j \frac{F(e_j)}{F(g_0)}} g_0
+ \sum_{j \in \Lambda'} \frac{a_j}{\beta_j} e'_j
\\
&= \frac{F(h)}{F(g_0)} \, g_0 + \sum_{j \in \Lambda'} \frac{a_j}{\beta_j} \, e'_j,
\end{align*}
then $h \in \overline{\text{span}\,\D}$, and $\D$ is a dictionary.
Note that $f \in \text{span}\{g_0, g_1\}$, and thus $f \in A_0(\D)$.
However, we claim that element $g_0$ does not approximate $f$, i.e.
$$
\text{arg}\min_{\mu} \n{f - \mu g_0} = 0.
$$
Indeed, for any $\mu > 0$
\begin{align*}
\n{f + \mu g_0} &\geq F(f + \mu g_0) = 1 + \mu F(g_0) > \n{f},
\\
\n{f - \mu g_0} &\geq F'(f - \mu g_0) = 1 - \mu F'(g_0) > \n{f}.
\end{align*}
Additionally, the choice of the elements~\eqref{not_smooth_g_0_g_1} 
and~\eqref{not_smooth_e'_j} of the dictionary $\D$ provides
\begin{align*}
F(g_0) &> 0,
\\
F(g_1) &= 0,
\\
F(e'_j) &= 0, \quad\text{for any } j \in \Lambda'.
\end{align*}
Then consider the following realization of WCGA of $f$:
for any $n \geq 1$ choose $F_{n-1} = F$, $\phi_n = g_0$, 
and $f_n = f$.
Hence $\n{f_n} \not\to 0$ and WCGA does not converge.
\end{proof}

Next we demonstrate that assumptions of Theorem~\ref{awcga_convergence} are 
weaker than the ones of Theorem~\ref{t_awcga_convergence}. 
Assume that conditions~\eqref{t_awcga_sum_t^p=infty}--\eqref{t_awcga_eta=o(t^p)}
hold. Then condition~\eqref{awcga_delta=o(t^p)} holds for any subsequence $\seq[k]{n}$. 
The following lemma shows that conditions~\eqref{t_awcga_sum_t^p=infty} and 
\eqref{t_awcga_eta=o(t^p)} imply conditions~\eqref{awcga_sum_t^p=infty} and 
\eqref{awcga_eta=o(t^p)} with some subsequence $\seq[k]{n}$.

\begin{lemma}
Let $\seq[n]{a}$ and $\seq[n]{b}$ be any such nonnegative sequences that
$$
\sm[n] b_n = \infty
\quad\text{and}\quad
a_n = o(b_{n-1}).
$$
Then there exists a subsequence $\seq[k]{n}$ such that
$$
\sm[k] b_{n_k} = \infty
\quad\text{and}\quad
a_{n_k} = o(b_{n_k}).
$$
\end{lemma}

\begin{proof}
We claim that there exists a subsequence $\seq[k]{n}$ such that 
\begin{equation}\label{proposition_assumption}
\frac{b_{n_k}}{b_{n_k-1}} \geq \frac{1}{2}
\quad\text{and}\quad
\sm[k] b_{n_k} = \infty.
\end{equation}
Indeed, let $\Lambda \subset \mathbb{N}$ denote the set of all such
indices $\lambda \in \mathbb{N}$ that
\begin{equation*}
\frac{b_{\lambda}}{b_{\lambda-1}} < \frac{1}{2}.
\end{equation*}
Then $\Lambda = \bigcup\limits_{k=1}^{N}  \Lambda_k$,
where $\seq[k][1][N]{\Lambda}$ are disjoint connected subsets of $\Lambda$ 
(i.e. if $\mu_1, \mu_2 \in \Lambda_k$, then for any 
$\lambda \in \mathbb{N}$ such that $\mu_1 < \lambda < \mu_2$,
$\lambda \in \Lambda_k$) for some $N \in \mathbb{N}\cup\{\infty\}$.
For each $k \in [1;\,N]$ define 
$\lambda_k = \min\limits_{\lambda\in\mathbb{N}}{\{\lambda \in \Lambda_k\}}$~-- 
the minimal element of the set $\Lambda_k$~--
and consider the sum $\sm[k][1][N] b_{\lambda_k}$.\\
If $\sm[k][1][N] b_{\lambda_k} = \infty$, 
take $\seq[k]{n} =  \{\lambda_k - 1\}_{k=2}^{\infty}$.
Note that in this case $N = \infty$ and the sequence $\seq[k][2]{\lambda}$
(and therefore $\seq[k]{n}$) is defined correctly.
It follows from the definition of $\lambda_k$ that 
$\lambda_k - 1 \not\in \Lambda$, hence conditions~\eqref{proposition_assumption}
hold for the sequence $\seq[k]{n}$.\\
If $\sm[k][1][N] b_{\lambda_k} < \infty$, then 
note that for any $\lambda \in \Lambda_k$ the definition of the set 
$\Lambda$ implies $b_{\lambda} < 2^{-(\lambda-\lambda_k)} b_{\lambda_k}$.
Hence
\begin{equation*}
\sum_{\lambda\in\Lambda} b_{\lambda} 
= \sum_{k=1}^{N} \sum_{\lambda\in\Lambda_k} b_{\lambda}
< \sum_{k=1}^{N} \sum_{j=1}^{|\Lambda_k|} 2^{-(j-1)} b_{\lambda_k}
\leq 2 \sum_{k=1}^{N} b_{\lambda_k}
< \infty.
\end{equation*}
Then take $\seq[k]{n} = \mathbb{N} \setminus \Lambda$
and note that for this sequence 
conditions~\eqref{proposition_assumption} hold.

Now choose any subsequence $\seq[k]{n}$ satisfying
conditions~\eqref{proposition_assumption}. 
Then
$$
\frac{a_{n_k}}{b_{n_k}} 
\leq 2\frac{a_{n_k}}{b_{n_k-1}} \to 0,
$$
hence $a_{n_k} = o(b_{n_k})$, which proves the lemma.
\end{proof}

We will use the following lemma (see Lemma~2.3 from~\cite{t_gtabsa}) 
to investigate convergence and the rate of approximation of the AWCGA.

\begin{referencelemma}\label{reference_lemma_E_m<E_n}
Let $X$ be a uniformly smooth Banach space with the modulus of
smoothness $\rho(u)$ and $\D$ be any dictionary. 
Take a number $\epsilon \geq 0$ and two elements 
$f$ and $f^{\epsilon}$ from $X$ such that $\n{f-f^{\epsilon}} \leq \epsilon$
and $f^{\epsilon}/A \in A_1(\D)$ with some number $A = A(\epsilon) > 0$.
Then for the AWCGA with a weakness sequence $\seq[n]{t}$, 
a perturbation sequence $\seq[n][0]{\delta}$, 
and a bounded error sequence $\seq[n]{\eta}$ 
\begin{equation*}
E_{n+1} \leq \n{f_n} \inf_{\lambda\geq0} 
\br{1 + \delta_n - \frac{\lambda t_{n+1}}{A} \br{1 - \delta_n 
- \frac{\beta_n + \epsilon}{\n{f_n}}} + 2\rho\br{\frac{\lambda}{\n{f_n}}}}
\end{equation*}
for any $n \geq 0$, where 
\begin{equation*}
\beta_n = \inf_{\mu\geq0} 
\frac{\delta_n + \eta_n + 2\rho\br{\mu(2 + \eta_n)\n{f}}}{\mu}.
\end{equation*}
\end{referencelemma}

We are now ready to prove Theorem~\ref{awcga_convergence}.
Note that the proof of sufficiency follows closely the proof of
Theorem~\ref{t_awcga_convergence} given in~\cite{t_gtabsa}.
One can acquire it by applying the estimates in Theorem~\ref{t_awcga_convergence} 
only for indices $\seq[k]{n}$ instead of every index $n \in \mathbb{N}$.
We present it here for completeness.

\begin{proof}[Proof of Theorem~\ref{awcga_convergence}]
First we prove the sufficiency of 
conditions~\eqref{awcga_sum_t^p=infty}--\eqref{awcga_eta=o(t^p)}.
Assume that for some $f \in X$ the AWCGA of $f$ does not converge to $f$, 
i.e. $\lim\limits_{n \to \infty} E_n = \alpha$ for some 
$0 < \alpha \leq \n{f}$. Then for any $n \geq 0$
\begin{equation}\label{awcga_convergence_f_n_estimate}
\alpha \leq E_n \leq \n{f_n}.
\end{equation}
Take $\epsilon = \alpha/2$ and find such element $f^{\epsilon} \in X$
and number $A > 0$ that $\n{f - f^{\epsilon}} \leq \epsilon$ 
and $f^{\epsilon}/A \in A_1(\D)$. 
Let $\{n_k\}_{k=1}^{\infty}$ be a subsequence,
for which assumptions of the theorem hold.
Then by Lemma~\ref{reference_lemma_E_m<E_n}
\begin{equation*}
E_{n_k+1} \leq \n{f_{n_k}} \inf_{\lambda\geq0} 
\br{1 + \delta_{n_k} - \frac{\lambda t_{{n_k}+1}}{A} \br{1 - \delta_{n_k} 
- \frac{\beta_{n_k} + \alpha/2}{\n{f_{n_k}}}} 
+ 2\rho\br{\frac{\lambda}{\n{f_{n_k}}}}},
\end{equation*}
where 
\begin{equation*}
\beta_{n_k} = \inf_{\mu\geq0} 
\frac{\delta_{n_k} + \eta_{n_k} + 2\rho\br{\mu(2 + \eta_{n_k})\n{f}}}{\mu}.
\end{equation*}
We first estimate $\beta_{n_k}$.
Using condition $\rho(u) \leq \gamma u^q$ we obtain
\begin{equation*}
\beta_{n_k} \leq \inf_{\mu\geq0} 
\br{2\gamma\br{(2 + \eta_0)\n{f}}^q \mu^{q-1} 
+ \br{\delta_{n_k} + \eta_{n_k}}\mu^{-1}}.
\end{equation*}
Consider the real valued function $\varphi(x) = ax^{q-1} + bx^{-1}$.
Then 
$$
\inf\limits_{x\geq0} \varphi(x) 
= \varphi\br{\br{\frac{b}{a(q-1)}}^{1/q}}
= p (q-1)^{1/q} a^{1/q} b^{1/p}
$$
and therefore
\begin{equation*}
\beta_{n_k} \leq
p (2\gamma(q-1))^{1/q} (2 + \eta_0) \n{f} \br{\delta_{n_k} + \eta_{n_k}}^{1/p}
= c \br{\delta_{n_k} + \eta_{n_k}}^{1/p},
\end{equation*}
where $c = p (2\gamma(q-1))^{1/q} (2 + \eta_0) \n{f}$.
Then, applying estimates $\rho(u) \leq \gamma u^q$ 
and~\eqref{awcga_convergence_f_n_estimate}, we get
\begin{equation*}
E_{n_k+1} \leq \n{f_{n_k}} \inf_{\lambda\geq0}
\br{1 + \delta_{n_k} - \frac{\lambda t_{n_k+1}}{A} 
\br{\frac{1}{2} - \delta_{n_k} 
- \frac{c}{\alpha} \br{\delta_{n_k} + \eta_{n_k}}^{1/p}}
+ 2\gamma \frac{\lambda^q}{\alpha^q}}.
\end{equation*}
Consider the real valued function $\psi(x) = ax^q - bx$.
Then
\begin{equation*}
\inf\limits_{x\geq0} \psi(x)
= \psi\br{\br{\frac{b}{aq}}^{1/(q-1)}}
= -(q-1) q^{-p} a^{-p/q} b^p.
\end{equation*}
Thus
\begin{equation*}
E_{n_k+1} \leq \n{f_{n_k}} \br{1 + \delta_{n_k} - B_k t_{n_k+1}^p},
\end{equation*}
where $B_k = (q-1) q^{-p} (2\gamma)^{-p/q} A^{-p} 
\br{\br{\frac{1}{2} - \delta_{n_k}} \alpha 
- c \br{\delta_{n_k} + \eta_{n_k}}^{1/p}}^p$.
Hence the estimate $E_{n_{k+1}} \leq E_{n_k+1}$ and 
condition~\eqref{awcga_G_n} provide
\begin{equation*}
E_{n_{k+1}} \leq E_{n_k} (1 + \eta_{n_k}) \br{1 + \delta_{n_k} - B_k t_{n_k+1}^p}.
\end{equation*}
Note that $B_k > 0$ for sufficiently big $k$ since
conditions~\eqref{awcga_delta=o(t^p)} and~\eqref{awcga_eta=o(t^p)}
imply $\delta_{n_k} \to 0$ and $\eta_{n_k} \to 0$.
Assume without loss of generality that $B_k \geq B > 0$ for all 
$k \in \mathbb{N}$ with some constant $B > 0$.
Then
\begin{equation}\label{awcga_convergence_E_n_recursive_estimate}
E_{n_{k+1}} \leq E_{n_k} (1 + \eta_{n_k}) \br{1 + \delta_{n_k} - B t_{n_k+1}^p}.
\end{equation}
We show that $B < 1$.
It is easy to see that the definition of the modulus 
of smoothness~\eqref{modulusofsmoothness} implies
$u-1 \leq \rho(u) \leq \gamma u^q$, hence $\gamma \geq 2^{-q}$.
Then, taking into account that 
$$
A \geq \n{f^{\epsilon}} \geq \n{f} - \frac{\alpha}{2} \geq \frac{\n{f}}{2},
$$ 
we obtain 
\begin{align*}
B \leq B_k &= (q-1) q^{-p} (2\gamma)^{-p/q} A^{-p} 
\br{\br{\frac{1}{2} - \delta_{n_k}} \alpha 
- c \br{\delta_{n_k} + \eta_{n_k}}^{1/p}}^p
\\ 
&\leq (q-1) q^{-p} (2\gamma)^{-p/q} A^{-p} \br{\frac{\alpha}{2}}^p
\\
&\leq (q-1) q^{-p} (2\gamma)^{-p/q} 
< 2 \br{1 - \frac{1}{q}}
\leq 1.
\end{align*}
Therefore $1 + \delta_{n_k} - B t_{n_k+1}^p > 0$ and recursively applying 
estimate~\eqref{awcga_convergence_E_n_recursive_estimate} provides
\begin{align*}
E_{n_{k+1}} 
&\leq \n{f} \prod_{j=1}^{k} (1 + \eta_{n_j})
\br{1 + \delta_{n_j} - B t_{n_j+1}^p}
\\
&\leq \n{f} \prod_{j=1}^{k} 
\br{1 + \delta_{n_j} + 2\eta_{n_j} - B t_{n_j+1}^p} 
< \alpha
\end{align*}
for sufficiently big $k$ by assumptions
$$
\delta_{n_k} = o(t_{n_k+1}^p),
\quad
\eta_{n_k} = o(t_{n_k+1}^p),
\quad\text{and}\quad
\sm[k] t_{n_k+1}^p = \infty.
$$
Hence $E_{n_{k+1}} < \alpha$, which contradicts
the assumption $\lim\limits_{n\to\infty} E_n = \alpha$. 
Therefore $\lim\limits_{n\to\infty} E_n = 0$ and
by~\eqref{awcga_G_n}
$$
\n{f_n} \leq (1+\eta_n)E_n \leq (1+\eta_0)E_n \to 0,
$$
i.e. the AWCGA of $f$ converges to $f$.

Now we prove the necessity of conditions~\eqref{awcga_sum_t^p=infty}--\eqref{awcga_eta=o(t^p)}.
Assume that for any subsequence $\{n_k\}_{k=1}^{\infty}$ at least one of the statements
\begin{align*}
&\sm[k] t_{n_k+1}^p < \infty,\\
&\delta_{n_k} \neq o(t_{n_k+1}^p),\\
&\eta_{n_k} \neq o(t_{n_k+1}^p),
\end{align*}
holds. 
We construct an example of a Banach space, a dictionary, and an element, 
for which the AWCGA with any weakness, perturbation, and error sequences, 
satisfying the stated assumption, does not converge.\\
For a number $\alpha > 0$ define sets
$\Lambda_1 = \{n \in \mathbb{N} : \delta_{n-1} \geq \alpha t^p_n 
\text{ or } \eta_{n-1} \geq \alpha t^p_n\}$
and $\Lambda_2 = \mathbb{N} \setminus \Lambda_1$.
We claim that there exists an $\alpha > 0$ such that
\begin{equation}
\label{awcga_convergence_necessity_alpha_claim}
\sum_{j \in \Lambda_2} t^p_j < \infty.
\end{equation}
Indeed, if $\sum\limits_{j \in \Lambda_2} t^p_j = \infty$ for any $\alpha > 0$ 
then for every $k \geq 1$ choose $\alpha(k) = 2^{-k}$, 
and find $\Gamma_k \subset \Lambda_2(k)$ such that  
$\min\{n : n \in \Gamma_k\} > \max\{n : n \in \Gamma_{k-1}\}$, 
$\max\{n : n \in \Gamma_k\} < \infty$, 
and $\sum\limits_{j \in \Gamma_k} t^p_j \geq 1$.
Hence by considering the union $\bigcup\limits_{k=1}^{\infty} \Gamma_k$,
we receive the sequence for which 
conditions~\eqref{awcga_sum_t^p=infty}--\eqref{awcga_eta=o(t^p)} hold,
which contradicts the aforementioned assumption.
\\
Take any $1 < q \leq 2$ and let $X = \ell_q$, and let $\mathcal{D} = \seq[n][0] {\pm e}$,
where $\seq[n][0] e$ is a standard basis in $\ell_q$.
Fix an $\alpha > 0$ such that claim~\eqref{awcga_convergence_necessity_alpha_claim}
holds, and find corresponding sets $\Lambda_1$ and $\Lambda_2$.
\\
If $|\Lambda_1| < \infty$ then $\sm[n] t^p_n < \infty$.
In this case it is known that for $f = e_0 + \sm[j] t^{p/q}_j e_j$
the WCGA (and therefore the AWCGA) does not converge to $f$ (see Proposition~2.1 from~\cite{t_gabs}).
\\
Consider the case $|\Lambda_1| = \infty$.
Without loss of generality assume that $1 \in \Lambda_2$. 
Choose a nonnegative sequence $\{a_j\}_{j \in \Lambda_1}$ such that
$\sum\limits_{j \in \Lambda_1} a^q_j = 1$ 
and $a_j \leq \alpha^{1/q}$ for any $j \in \Lambda_1$, 
and take 
$$
f = \sum\limits_{j \in \Lambda_1} a_j e_j 
+ \alpha^{1/q} \sum\limits_{j \in \Lambda_2} t^{p/q}_j e_j.
$$
We claim that there exists a realization of the AWCGA of $f$ such that 
for any $n \geq 1$
\begin{equation}
\label{awcga_convergence_awcga_claim}
f_n = \eta_n^{1/q} e_1 + \sum\limits_{j \in \Lambda_1} a_j e_j 
+ \alpha^{1/q} \sum\limits_{j \in \Lambda_2^{(n)}} t^{p/q}_j e_j,
\end{equation}
where $\Lambda_2^{(n)} = \Lambda_2 \setminus \Gamma_n$, 
and $\Gamma_n$ is a set of indices of $e_j$ chosen on the first $n$ steps of the algorithm.
\\
\indent For $n = 1$ choose 
$$
F_0(x) = F_f(x) 
= \frac{\sum\limits_{j \in \Lambda_1} a_j^{q/p} x_j 
+ \alpha^{1/p} \sum\limits_{j \in \Lambda_2} t_j x_j}
{\n{f}_q^{q/p}}.
$$
Then
\begin{align*}
F_0(e_0) &= 0,
\\
F_0(e_j) &= a_j^{q/p} \n{f}_q^{-q/p} \leq \alpha^{1/p} \n{f}_q^{-q/p}
\text{ for any } j \in \Lambda_1,
\\
F_0(e_j) &= t_j \alpha^{1/p} \n{f}_q^{-q/p}
\text{ for any } j \in \Lambda_2.
\end{align*}
Therefore the choice $\phi_1 = e_1$ is possible since $1 \in \Lambda_2$.
Thus $\Gamma_1 = \{1\}$ and taking 
$$
f_1 = \eta_1^{1/q} e_1 + \sum\limits_{j \in \Lambda_1} a_j e_j 
+ \alpha^{1/q} \sum\limits_{j \in \Lambda_2^{(1)}} t^{p/q}_j e_j
$$
satisfies~\eqref{awcga_G_n} since 
$E_1 = \br{1 + \alpha \sum\limits_{j \in \Lambda_2^{(1)}} t^p_j}^{1/q} \geq 1$ and
\begin{align*}
\n{f_1}_q^q &= \eta_1 + \br{1 + \alpha \sum\limits_{j \in \Lambda_2^{(1)}} t^p_j}
\\
&\leq (1 + \eta_1) E_1^q
\\
&\leq (1 + \eta_1)^q E_1^q.
\end{align*}
Hence for $n = 1$ claim~\eqref{awcga_convergence_awcga_claim} holds.
\\
\indent For $n > 1$, 
provided $f_{n-1} = \eta_{n-1}^{1/q} e_1 + \sum\limits_{j \in \Lambda_1} a_j e_j 
+ \alpha^{1/q} \sum\limits_{j \in \Lambda_2^{(n-1)}} t^{p/q}_j e_j$, choosing 
$$
F_{n-1}(x) = 
\frac{\delta_{n-1}^{1/p} x_0 + \eta_{n-1}^{1/p} x_1 
+ \sum\limits_{j \in \Lambda_1} a_j^{q/p} x_j 
+ \alpha^{1/p} \sum\limits_{j \in \Lambda_2^{(n-1)}} t_j x_j}
{(1 + \delta_{n-1})^{1/p} \n{f_{n-1}}_q^{q/p}}
$$
satisfies~\eqref{awcga_F_n} since 
$$
|F_{n-1}(x)| \leq \frac{\br{\delta_{n-1} + \eta_{n-1} 
+ \sum\limits_{j \in \Lambda_1} a_j^q 
+ \alpha \sum\limits_{j \in \Lambda_2^{(n-1)}} t_j^p}^{1/p} 
\br{\sm[j][0] x_j^q}^{1/q}}
{(1 + \delta_{n-1})^{1/p} \n{f_{n-1}}_q^{q/p}}
\leq \n{x}_q
$$
and
$$
F_{n-1}(f_{n-1}) = \frac{\eta_{n-1} 
+ \sum\limits_{j \in \Lambda_1} a_j^q 
+ \alpha \sum\limits_{j \in \Lambda_2^{(n-1)}} t_j^p}
{(1 + \delta_{n-1})^{1/p} \n{f_{n-1}}_q^{q/p}} 
= \frac{\n{f_{n-1}}_q}{(1 + \delta_{n-1})^{1/p}}
\geq (1 - \delta_{n-1}) \n{f_{n-1}}_q,
$$
where the last inequality uses the estimate
$$
(1 + \delta_{n-1})^{1/p} (1 - \delta_{n-1})  
= \br{1 - \delta^2_{n-1}}^{1/p} \br{1 - \delta_{n-1}}^{1/q} 
\leq 1.
$$
Hence such choice of a functional is possible.
Let $A_n = \br{(1 + \delta_{n-1}) \n{f_{n-1}}_q^q}^{-1/p}$.
Then
\begin{align*}
F_{n-1}(e_0) &= \delta_{n-1}^{1/p} A_n,
\\
F_{n-1}(e_1) &= \eta_{n-1}^{1/p} A_n,
\\
F_{n-1}(e_j) &= a_j^{q/p} A_n \leq \alpha^{1/p} A_n
\text{ for any } j \in \Lambda_1,
\\
F_{n-1}(e_j) &= t_j \alpha^{1/p} A_n
\text{ for any } j \in \Lambda_2^{(n-1)},
\\
F_{n-1}(e_j) &= 0
\text{ for any } j \in \Gamma_{n-1} \setminus \{1\}.
\end{align*}
If $n \in \Lambda_2$ we choose $\phi_n = e_n$.
Otherwise $n \in \Lambda_1$, and either 
$$
F_{n-1}(e_0) \geq t_n \alpha^{1/p} A_n
\quad\text{and/or}\quad
F_{n-1}(e_1) \geq t_n \alpha^{1/p} A_n.
$$
Therefore it is possible to choose $\phi_n = e_0$ or $\phi_n = e_1$.
In any case $\Gamma_n \cap \Lambda_1 = \varnothing$ and taking 
$$
f_n = \eta_n^{1/q} e_1 + \sum\limits_{j \in \Lambda_1} a_j e_j 
+ \alpha^{1/q} \sum\limits_{j \in \Lambda_2^{(n)}} t^{p/q}_j e_j
$$
satisfies~\eqref{awcga_G_n} since 
$E_n = \br{1 + \alpha \sum\limits_{j \in \Lambda_2^{(n)}} t^p_j}^{1/q} \geq 1$ and
\begin{align*}
\n{f_n}_q^q &= \eta_1 + \br{1 + \alpha \sum\limits_{j \in \Lambda_2^{(n)}} t^p_j}
\\
&\leq (1 + \eta_n) E_n^q
\\
&\leq (1 + \eta_n)^q E_n^q.
\end{align*}
Hence claim~\eqref{awcga_convergence_awcga_claim} holds for any $n > 1$.
Thus $\n{f_n} \not\to 0$ and the AWCGA does not converge to $f$.
\end{proof}

The next corollary gives the particular rates for the weakness, 
perturbation, and error sequences, which are sufficient for convergence.

\begin{corollary}
Let $X \in \mathcal{P}_q$ be a Banach space and $\D$ be any dictionary.
Let $\seq[n]{t}$, 
$\seq[n][0]{\delta}$, and $\seq[n]{\eta}$ be any such sequences that
for some subsequence $\seq[k]{n}$
$$
t_{n_k} \asymp k^{-1/p},
\quad
\delta_{n_k} \asymp k^{-r},
\quad\text{and}\quad
\eta_{n_k} \asymp k^{-s},
$$
where $p = q/(q-1)$ and $r, s > 1$.
Then the AWCGA with the weakness sequence $\seq[n]{t}$, the perturbation
sequence $\seq[n][0]{\delta}$, and the bounded error sequence 
$\seq[n]{\eta}$ converges for any $f \in X$.
\end{corollary}

We now justify the restriction imposed on an error sequence in 
Theorem~\ref{awcga_convergence}.
The following proposition demonstrates that if an error sequence is
unbounded, then for any $1 < q \leq 2$ there exists 
a dictionary $\D$ and an element $f$ in $\ell_q$ such that
the AWCGA of $f$ diverges.

\begin{proposition}
For any $\seq[n]{\eta}$ with $\limsup\limits_{n\to\infty} \eta_n = \infty$
and any $1 < q \leq 2$ there exists a dictionary $\D$ in $\ell_q$ and 
an element $f \in \ell_q$ such that AWCGA of $f$ 
with any weakness sequence $\seq[n]{t}$, 
any perturbation sequence $\seq[n][0]{\delta}$, 
and the error sequence $\seq[n]{\eta}$ does not converge to $f$.
\end{proposition}

\begin{proof}
Let $\D = \seq[n]{\pm e}$, 
where $\seq[n]{e}$ is a standard basis in $\ell_q$.
Choose any weakness sequence $\seq[n]{t}$,
a perturbation sequence $\seq[n][0]{\delta}$ and 
find a subsequence $\seq[k]{n}$
such that $\eta_{n_k} \geq k$ for any $k \geq 1$.
Take a positive nonincreasing sequence $\seq[j]{a} \in \ell_q$ 
such that $\sm a_j^q = 1$ and
$$
\br{\sm[j][n_k+1][\infty] a_j^q}^{1/q} \geq (k+1)^{-1}
$$
for any $k \geq 1$.
Denote $f = \sm a_j e_j \in \ell_q$.
\\
Consider the following realization of the AWCGA of $f$:
\\
\indent For $n \not\in \seq[k]{n}$ choose $\phi_n = e_n$ 
and $G_n = \sm[j][1][n] a_j e_j$.
\\
\indent For $n \in \seq[k]{n}$ choose $\phi_n = e_n$
and $G_n = 0$.
\\
Then for any $k \geq 1$ norm of the remainder 
$\n{f_{n_k}} = \n{f}$, hence $\n{f_n} \not\to 0$.
\end{proof}

The following lemma is used to prove Theorem~\ref{awcga_convergence_eta_delta_in_l_1}.

\begin{lemma}\label{lemma_sequences_l_1}
Let $\seq[n]{a}$ and $\seq[n]{b}$ be any such nonnegative sequences that
$$
\sm[n] a_n < \infty
\quad\text{and}\quad
\sm[n] b_n = \infty.
$$
Then there exists a subsequence $\seq[k]{n}$ such that
$$
\sm[k] b_{n_k} = \infty
\quad\text{and}\quad
a_{n_k} = o(b_{n_k}).
$$
\end{lemma}

\begin{proof}
Take $\Gamma_0 \subset \mathbb{N}$ -- the set of all such indices 
$\lambda \in \mathbb{N}$ that $a_{\lambda} = 0$.
If $\sum\limits_{\lambda\in\Gamma_0} b_{\lambda} = \infty$,
then $|\Gamma_0| = \infty$, and $\seq[k]{n} = \Gamma_0$
is a required subsequence.\\
Consider the case $\sum\limits_{\lambda\in\Gamma_0} b_{\lambda} < \infty$.
Take $\Gamma_1 \subset \mathbb{N}$ -- the set of all such indices 
$\lambda \in \mathbb{N}$ that $\frac{a_{\lambda}}{b_{\lambda}} > 1$ 
or $b_{\lambda} = 0$. Therefore 
$\sum\limits_{\lambda\in\Gamma_1} b_{\lambda}
< \sum\limits_{\lambda\in\Gamma_1} a_{\lambda}  
\leq \sm[n] a_n < \infty$.\\
Set $\Lambda_0 = \Gamma_0 \cup \Gamma_1$. Then 
$$
\sum\limits_{\lambda\in\Lambda_0} b_{\lambda} \leq 
\sum\limits_{\lambda\in\Gamma_0} b_{\lambda}
+ \sum\limits_{\lambda\in\Gamma_1} b_{\lambda}
< \infty.
$$
For each $k \in \mathbb{N}$ denote by 
$\Lambda_k \subset \mathbb{N}\setminus\Lambda_0$ the set 
of all such indices $\lambda \in \mathbb{N}$ that
\begin{equation}\label{definition_Lambda_k}
\frac{1}{k+1} < \frac{a_{\lambda}}{b_{\lambda}} \leq \frac{1}{k}.
\end{equation}
Note that for any $k \in \mathbb{N}$
$$
\sum\limits_{\lambda\in\Lambda_k} b_{\lambda}
< (k+1) \sum\limits_{\lambda\in\Lambda_k} a_{\lambda}
\leq (k+1) \sm[n] a_n < \infty,
$$
which implies that infinitely many sets $\Lambda_k$ are not empty 
and $\mathbb{N} = \bigcup\limits_{k=0}^{\infty} \Lambda_k$.
Thus
\begin{equation}\label{Lambda_sum_b=infty}
\sm[k]\sum_{\lambda\in\Lambda_k} b_{\lambda}^p = \infty.
\end{equation}
For each $k \in \mathbb{N}$ choose $\Omega_k \subset \Lambda_k$ such that
$$
|\Omega_k| < \infty
\quad\text{and}\quad
\sum_{\lambda\in\Omega_k} b_{\lambda}
\geq \frac{1}{2} \sum_{\lambda\in\Lambda_k} b_{\lambda}
$$
and set $\omega_k = \max\limits_{\lambda\in\mathbb{N}}
{\{\lambda \in \Omega_k\}} < \infty$.\\
We claim that subsequence $\seq[k]{n} = \bigcup\limits_{k=1}^{\infty} \Omega_k$
satisfies the requirements.
Indeed, by the choice of $\Omega_k$ and from~\eqref{Lambda_sum_b=infty}
$$
\sm[k] b_{n_k} 
= \sm[k]\sum_{\lambda\in\Omega_k} b_{\lambda} 
\geq \frac{1}{2} \sm[k]\sum_{\lambda\in\Lambda_k} b_{\lambda}
= \infty,
$$
hence the first requirement is satisfied.
We show that $a_{n_k} = o(b_{n_k})$.
Choose any $0 < \epsilon < 1$ and find $m \in \mathbb{N}$ such that
$$
\frac{1}{m+1} < \epsilon \leq \frac{1}{m}.
$$
Then for any $k > \max\limits_{1 \leq j \leq m} \{\omega_j\}$ 
by~\eqref{definition_Lambda_k}
$$
\frac{a_{n_k}}{b_{n_k}} 
\leq \frac{1}{m+1}
< \epsilon,
$$
hence the second requirement is satisfied.
\end{proof}

\section{Rate of convergence of the AWCGA}\label{rate_of_convergence_of_the_awcga}

In conclusion, we discuss the rate of convergence of the AWCGA.
It is clear that in order to get a nontrivial rate of approximation,
an additional requirement has to be imposed on an element.
Traditionally for this area, we restrict to the elements 
from the class $A_1(\D)$~-- the closure of the convex hull of $\D$.
We start with the known result for the rate of convergence 
of the WCGA (see Theorem~2.2 from~\cite{t_gabs}).

\begin{referencetheorem}\label{t_wcga_rate}
Let $X \in \mathcal{P}_q$ be a Banach space and $\D$ be any dictionary. 
Let $\seq[n]{t}$ be a weakness sequence.
Then for any $f \in A_1(\D)$ the AWCGA satisfies the estimate
\begin{equation}\label{t_wcga_rate_estimate}
\n{f_n} \leq C \br{1 + \sm[k][1][n] t_k^p}^{-1/p},
\end{equation}
where $p = q/(q-1)$ and $C = C(q, \gamma)$.
\end{referencetheorem}

The next result states the rate of convergence of an adaptive AWCGA,
where adaptive means that perturbation and error sequences are 
determined by the AWCGA applied to a given element $f \in A_1(\D)$ 
(see Theorem~2.4 from~\cite{t_gtabsa}). 

\begin{referencetheorem}\label{t_awcga_rate}
Let $X \in \mathcal{P}_q$ be a Banach space and $\D$ be any dictionary.
Let $\seq[n]{t}$ be a weakness sequence.
Then for any $f \in A_1(\D)$ the AWCGA
with the perturbation sequence $\seq[n][0]{\delta}$ and
the error sequence $\seq[n]{\eta}$, which are given by
\begin{align}
\label{t_awcga_rate_delta_condition}
\delta_n &= t_{n+1}^p \n{f_n}^p 3^{-p} 
\br{64\br{8\gamma}^{p/q}}^{-1},
\quad n \geq 0;
\\
\label{t_awcga_rate_eta_condition}
\eta_n &= t_{n+1}^p E_n^p 3^{-p} 
\br{64\br{8\gamma}^{p/q}}^{-1},
\quad n \geq 1,
\end{align}
where $p = q/(q-1)$, satisfies the estimate
\begin{equation*}
\n{f_n} \leq C \br{1 + \sm[k][1][n] t_k^p}^{-1/p},
\end{equation*}
where $C = C(q, \gamma)$.
\end{referencetheorem}

The particular advantage of this result is that the rate of convergence 
of the AWCGA with such perturbation and error sequences is the same as 
the rate of convergence of the WCGA~\eqref{t_wcga_rate_estimate}.\\
We prove the following theorem, which is essentially 
Theorem~\ref{t_awcga_rate} but with 
conditions~\eqref{t_awcga_rate_delta_condition} 
and~\eqref{t_awcga_rate_eta_condition} imposed only on 
some subsequence $\seq[k]{n}$. 
The proof we give here is basically the proof of Theorem~2.4 
from~\cite{t_gtabsa} with some minor distinctions.

\begin{theorem}\label{awcga_rate}
Let $X \in \mathcal{P}_q$ be a Banach space and $\D$ be any dictionary.
Let $\seq[n]{t}$ be a weakness sequence.
Then for any $f \in A_1(\D)$ the AWCGA
with a perturbation sequence $\seq[n][0]{\delta}$
and a bounded error sequence $\seq[n]{\eta}$ such that
for some subsequence $\seq[k]{n}$
\begin{align}
\label{awcga_rate_delta_condition}
\delta_{n_k} &= t_{n_k+1}^p \n{f_{n_k}}^p 3^{-p} 
\br{64\br{8\gamma}^{p/q}}^{-1},
\quad k \geq 1;
\\
\label{awcga_rate_eta_condition}
\eta_{n_k} &= t_{n_k+1}^p E_{n_k}^p 3^{-p} 
\br{64\br{8\gamma}^{p/q}}^{-1},
\quad k \geq 1,
\end{align}
where $p = q/(q-1)$, satisfies the estimate
\begin{equation*}
\n{f_n} \leq C \br{1 + \sm[k][1][N] t_{n_k}^p}^{-1/p},
\end{equation*}
where $C = C(q, \gamma, \eta_0) = 8(1+\eta_0)\gamma^{1/q}$ and 
$N = N(n) = \max\{k\in\mathbb{N} \, | \, n_k \leq n\}$.
\end{theorem}

Note that once a subsequence $\seq[k]{n}$, 
for which conditions~\eqref{awcga_rate_delta_condition} 
and~\eqref{awcga_rate_eta_condition} hold, is found,
only the choice of elements $\phi_{n_k}$ is essential for the
rate of convergence, 
so arbitrary elements $\phi_j$ can be chosen on other steps.

\begin{proof}[Proof of Theorem~\ref{awcga_rate}]
Take any $f \in A_1(\D)$. 
Then Lemma~\ref{reference_lemma_E_m<E_n} applied with the subsequence 
$\seq[k]{n}$ provides with $\epsilon = 0$ and $A = 1$
\begin{equation*}
E_{n_{k+1}} \leq \n{f_{n_k}} \inf_{\lambda\geq0} 
\br{1 + \delta_{n_k} - \lambda t_{n_k+1} \br{1 - \delta_{n_k} 
- \frac{\beta_{n_k}}{\n{f_{n_k}}}} + 2\rho\br{\frac{\lambda}{\n{f_{n_k}}}}},
\end{equation*}
where 
\begin{equation*}
\beta_{n_k} = \inf_{\mu\geq0} 
\frac{\delta_{n_k} + \eta_{n_k} + 2\rho\br{\mu(2 + \eta_{n_k})\n{f}}}{\mu}.
\end{equation*}
Note that since $\n{f} \leq 1$ and $\gamma \geq 2^{-q}$
\begin{align*}
\eta_{n_k} 
&\leq 3^{-p} \br{64(8 \, 2^{-q})^{p/q}}^{-1}
= \br{\frac{2}{3}}^p \frac{1}{8 \, 8^p}
< \frac{1}{2},
\\
\delta_{n_k} 
&\leq (1 + \eta_{n_k})^p 3^{-p} \br{64(8 \, 2^{-q})^{p/q}}^{-1}
< \frac{1}{8 \, 8^p}
< \frac{1}{16}.
\end{align*}
Hence, using an estimate $\rho(u) \leq \gamma u^q$,
and taking $\mu = \n{f_{n_k}}^{p/q} 3^{-p} \br{4(8\gamma)^{p/q}}^{-1}$
we obtain
\begin{align*}
\beta_{n_k} 
&\leq \br{\delta_{n_k} + \eta_{n_k}} \mu^{-1} 
+ 2\gamma \mu^{q-1} (2 + \eta_{n_k})^q \n{f}^q
\\
&= \frac{1}{16} \n{f_{n_k}} + \br{\frac{2+\eta_n}{12}}^q \n{f_{n_k}}
< \frac{3}{8} \n{f_{n_k}}.
\end{align*}
Thus, since $\rho(u) \leq \gamma u^q$ and $\delta_{n_k} < 1/16$,
\begin{equation*}
E_{n_{k+1}} \leq \n{f_{n_k}} \inf_{\lambda\geq0} 
\br{1 + \delta_{n_k} - \frac{9}{16} \lambda t_{n_k+1} 
+ 2\gamma \lambda^q \n{f_{n_k}}^{-q}}.
\end{equation*}
Take $\lambda = t_{n_k+1}^{p/q} \n{f_{n_k}}^p (8\gamma)^{-p/q}$.
Then using~\eqref{awcga_G_n} and 
conditions~\eqref{awcga_rate_delta_condition} we get
\begin{align*}
E_{n_{k+1}} 
&\leq \n{f_{n_k}} \br{1 + \delta_{n_k} 
- \frac{5}{16} \frac{t_{n_k+1}^p \n{f_{n_k}}^p}{(8\gamma)^{p/q}}}
\leq \n{f_{n_k}} \br{1 - \frac{t_{n_k+1}^p \n{f_{n_k}}^p}{4 (8\gamma)^{p/q}}}
\\
&\leq E_{n_k} \br{1 + \frac{t_{n_k+1}^p E_{n_k}^p}{3^p \, 64 (8\gamma)^{p/q}}}
\br{1 - \frac{t_{n_k+1}^p E_{n_k}^p}{4 (8\gamma)^{p/q}}}
< E_{n_k} \br{1 - \alpha t_{n_k+1}^p E_{n_k}^p},
\end{align*}
where $\alpha = \br{8 (8\gamma)^{p/q}}^{-1}$.
Then, since $x^p \leq x$ for $0 \leq x \leq 1$,
\begin{equation}\label{awcga_rate_E_n^p_estimate}
E_{n_{k+1}}^p \leq E_{n_k}^p \br{1 - \alpha t_{{n_k}+1}^p E_{n_k}^p}.
\end{equation}
Using the technique from Lemma~2.16 from~\cite{t_ga},
we will show that for any $m \geq 0$
\begin{equation}\label{awcga_rate_E_n_induction}
E_{n_m} \leq \alpha^{-1/p} \br{1 + \sm[k][1][m] t_{n_k}^p}^{-1/p}.
\end{equation}
We use induction by $m$ to prove~\eqref{awcga_rate_E_n_induction}.
First, note that $\alpha < 1$ since $\gamma \geq 2^{-q}$.
Then for $m = 0$
$$
E_0 = \n{f} \leq 1 < \alpha^{-1/p},
$$
hence the estimate holds. Assume that it holds for $m$.
Then inequality~\eqref{awcga_rate_E_n^p_estimate} provides
\begin{align*}
E_{n_{m+1}}^{-p} 
&\geq E_{n_m}^{-p} \br{1 - \alpha t_{{n_m}+1}^p E_{n_m}^p}^{-1}
\geq E_{n_m}^{-p} \br{1 + \alpha t_{{n_m}+1}^p E_{n_m}^p}
\\
&= E_{n_m}^{-p} + \alpha t_{{n_m}+1}^p
\geq \alpha \br{1 + \sm[k][1][m+1] t_{n_k}^p},
\end{align*}
since by the assumption estimate~\eqref{awcga_rate_E_n_induction} holds for $m$.
Thus the induction holds and estimate~\eqref{awcga_rate_E_n_induction} 
is correct for any $m \geq 0$. Then by~\eqref{awcga_G_n}
$$
\n{f_n}
\leq (1+\eta_n) E_n \leq (1+\eta_n) E_{n_N}
\leq C \br{1 + \sm[k][1][N] t_{n_k}^p}^{-1/p},
$$
where $C = (1 + \eta_0) \alpha^{-1/p} = 8 (1 + \eta_0) \gamma^{1/q}$.
\end{proof}

The following corollary shows that the AWCGA will converge with the same
rate as the WCGA as long as adequately precise computations are made 
sufficiently often. 

\begin{corollary}
Let $X \in \mathcal{P}_q$ be a Banach space and $\D$ be any dictionary.
Let $\seq[n]{t}$, $\seq[n][0]{\delta}$ and $\seq[n]{\eta}$ 
be any such sequences that for some subsequence $\seq[k]{n}$ 
with $M = \sup\limits_{k\in\mathbb{N}} \m{n_{k+1} - n_k} < \infty$
\begin{align*}
&t_{n_k+1} \geq t > 0,
\\
&\delta_{n_k} \leq t^p \n{f_{n_k}}^p 3^{-p} 
\br{64\br{8\gamma}^{p/q}}^{-1},
\\
&\eta_{n_k} \leq t^p E_{n_k}^p 3^{-p} 
\br{64\br{8\gamma}^{p/q}}^{-1},
\end{align*}
where $p = q/(q-1)$. Then for any $f \in A_1(\D)$ the AWCGA
with the weakness sequence $\seq[n]{t}$, the perturbation sequence 
$\seq[n][0]{\delta}$, and the bounded error sequence $\seq[n]{\eta}$ 
satisfies the estimate
\begin{equation*}
\n{f_n} \leq C n^{-1 + 1/q},
\end{equation*}
where $C = C(q, \gamma, \eta_0, M)$.
\end{corollary}

We also state two corollaries that give the rate of convergence of the 
adaptive AWCGA for the standard dictionary in $\ell_p$ spaces.

\begin{corollary}
Let $X = \ell_q$, $1 < q \leq 2$ and $\D = \seq[n]{\pm e}$, 
where $\seq[n]{e}$ is a standard basis in $\ell_q$. 
Let $a = \seq[n]{a} \in \ell_1$ be a positive nonincreasing sequence
with $\n{a}_1 \leq 1$. 
Let $\seq[n]{t}$, $\seq[n][0]{\delta}$, and $\seq[n]{\eta}$ 
be any such sequences that for some subsequence $\seq[k]{n}$ 
with $M = \sup\limits_{k\in\mathbb{N}} \m{n_{k+1} - n_k} < \infty$
\begin{align*}
t_{n_k+1} &\geq t > 0,
\\
\delta_{n_k} &\leq \frac{1}{8} \br{\frac{t}{24}}^{p} 
\br{\sm[j][n_k+1][\infty] a_j^q}^{p/q},
\\
\eta_{n_k} &\leq \frac{1}{8} \br{\frac{t}{24}}^{p} 
\br{\sm[j][n_k+1][\infty] a_j^q}^{p/q}.
\end{align*}
Then for the element $f = \sm[n] a_n e_n$ the AWCGA
with the weakness sequence $\seq[n]{t}$, the perturbation 
sequence $\seq[n][0]{\delta}$, and the bounded error 
sequence $\seq[n]{\eta}$ satisfies the estimate
\begin{equation*}
\n{f_n} \leq C n^{-1 + 1/q},
\end{equation*}
where $C = C(q, t, \eta_0, M)$.
\end{corollary}

\begin{corollary}
Let $X = \ell_q$, $2 \leq q < \infty$ and $\D = \seq[n]{\pm e}$, 
where $\seq[n]{e}$ is a standard basis in $\ell_q$. 
Let $a = \seq[n]{a} \in \ell_1$ be a positive nonincreasing sequence
with $\n{a}_1 \leq 1$. 
Let $\seq[n]{t}$, $\seq[n][0]{\delta}$, and $\seq[n]{\eta}$ 
be any such sequences that for some subsequence $\seq[k]{n}$ 
with $M = \sup\limits_{k\in\mathbb{N}} \m{n_{k+1} - n_k} < \infty$
\begin{align*}
t_{n_k+1} &\geq t > 0,
\\
\delta_{n_k} &\leq \frac{t^2}{4608} \br{\sm[j][n_k+1][\infty] a_j^q}^{2/q},
\\
\eta_{n_k} &\leq \frac{t^2}{4608} \br{\sm[k][n_k+1][\infty] a_j^q}^{2/q}.
\end{align*}
Then for the element $f = \sm[n] a_n e_n$ the AWCGA
with the weakness sequence $\seq[n]{t}$, the perturbation 
sequence $\seq[n][0]{\delta}$, and the bounded error 
sequence $\seq[n]{\eta}$ satisfies the estimate
\begin{equation*}
\n{f_n} \leq C n^{-1/2},
\end{equation*}
where $C = C(q, t, \eta_0, M)$.
\end{corollary}

\noindent{\bf Acknowledgements.}
I am grateful to V.\,N.~Temlyakov for his guidance during the writing 
of this paper and his support throughout the research process.

\bigskip
\bibliographystyle{abbrv}
\bibliography{bibliography}

\begin{thebibliography}{1}

\bibitem{beauzamy_ibstg}
B.~{Beauzamy}.
\newblock {\em {I}ntroduction to {B}anach spaces and their geometry}.
\newblock North-Holland Publishing Company, 1982.

\bibitem{dkt_csgabs}
S.~J. {Dilworth}, D.~{Kutzarova}, and V.~N. {Temlyakov}.
\newblock {C}onvergence of some greedy algorithms in {B}anach spaces.
\newblock {\em Journal of Fourier Analysis and Applications}, 8(5):\,489--506,
  2002.

\bibitem{donahue_et_al}
M.~J. Donahue, C.~Darken, L.~Gurvits, and E.~Sontag.
\newblock {R}ates of convex approximation in non-{H}ilbert spaces.
\newblock {\em Constructive Approximation}, 13(2):\,187--220, 1997.

\bibitem{dubinin_gaa}
V.~V. {Dubinin}.
\newblock {\em {G}reedy algorithms and applications}.
\newblock Doctoral Dissertation. University of South Carolina, 1997.

\bibitem{enflo_bswcgeuscn}
P.~{Enflo}.
\newblock {B}anach spaces which can be given an equivalent uniformly convex
  norm.
\newblock {\em Israel Journal of Mathematics}, 13(3):\,281--288, 1972.

\bibitem{t_gabs}
V.~N. {Temlyakov}.
\newblock {G}reedy algorithms in {B}anach spaces.
\newblock {\em Advances in Computational Mathematics}, 14(3):\,277--292, 2001.

\bibitem{t_gtabsa}
V.~N. {Temlyakov}.
\newblock {G}reedy-type approximation in {B}anach spaces and applications.
\newblock {\em Constructive Approximation}, 21(2):\,257--292, 2005.

\bibitem{t_ga}
V.~N. {Temlyakov}.
\newblock {\em {G}reedy approximation}.
\newblock Cambridge University Press, 2011.

\end{thebibliography}

\end{document}